\documentclass[11pt, reqno]{article}

\usepackage{fullpage}

\usepackage{amsmath,amsfonts,amssymb,graphics,amsthm}
\usepackage{hyperref}
\hypersetup{
    colorlinks=true,
    linkcolor=blue,
    citecolor=red,
    urlcolor=blue,
    pdfborder={0 0 0}
}

\usepackage{comment}
\usepackage{cleveref}
  \crefname{theorem}{Theorem}{Theorems}
  \crefname{theorem}{Theorem}{Theorems}
  \crefname{lemma}{Lemma}{Lemmas}
  \crefname{lem}{Lemma}{Lemmas}
  \crefname{remark}{Remark}{Remarks}
  \crefname{prop}{Proposition}{Propositions}
  \crefname{defn}{Definition}{Definitions}
  \crefname{corollary}{Corollary}{Corollaries}
  \crefname{section}{Section}{Sections}
  \crefname{figure}{Figure}{Figures}

\newtheorem{thm}{Theorem}[section]
\newtheorem{lemma}[thm]{Lemma}

\newtheorem{prop}[thm]{Proposition}

\numberwithin{equation}{section}

\theoremstyle{definition}
\newtheorem{remark}[thm]{Remark}

\numberwithin{equation}{section}

\def\cN{\mathcal{N}}
\def\cM{\mathcal{M}}

\def\cF{\mathcal{F}}

\def\cA{\mathcal{A}}

\def\Q{\mathbb{Q}}

\def\P{\mathbb{P}}
\def\E{\mathbb{E}}

\def\R{\mathbb{R}}
\def\Z{\mathbb{Z}}

\def  \p- {p\textunderscore}

\def\eps{\varepsilon}

\DeclareMathOperator{\var}{Var}

\DeclareMathOperator{\cov}{Cov}

\newcommand{\indic}[1]{\mathbf{1}_{\{#1\}}}

\begin{document}

\title{An elementary approach to Gaussian multiplicative chaos} 

\author{
Nathana\"el Berestycki\footnote{Statistical Laboratory, DPMMS, University of Cambridge.
\texttt{beresty@statslab.cam.ac.uk}. Supported in part by EPSRC grants EP/L018896/1 and EP/I03372X/1.
}
}

\maketitle




\begin{abstract}
A completely elementary and self-contained proof of convergence of Gaussian multiplicative chaos is given. The argument shows further that the limiting random measure is nontrivial in the entire subcritical phase $(\gamma < \sqrt{2d})$ and that the limit is universal (i.e., the limiting measure is independent of the regularisation of the underlying field).
\end{abstract}



\section{Introduction}

Gaussian multiplicative chaos is a theory initiated by Kahane \cite{Kahane}, whose goal (in a slightly updated language) is the definition and study of random measures of the form
\begin{equation}\label{goal}
\mu(dx) = e^{\gamma h(x) - \frac12 \gamma^2 \E( h(x)^2)} \sigma(dx)
\end{equation}
where $h$ is a centred Gaussian generalised random field subject to certain assumptions, $\gamma>0$, and $\sigma$ is a given reference measure on a domain $D$ of Euclidean space. Since $h$ is not defined pointwise but exists only as a distribution, it is not clear what meaning to give to \eqref{goal} a priori. In fact, some regularisation of the field and a suitable renormalisation have to be performed in order to construct $\mu$. The theory has generated considerable renewed interest notably because of its connection with two-dimensional Liouville Quantum Gravity and the KPZ relations. This is the particular case of the theory when $d=2$ and $h$ is the massless Gaussian free field or GFF (see \cite{Sheffield, GFFnotes}) with appropriate boundary conditions. The paper by Duplantier and Sheffield \cite{DuplantierSheffield}
constructed the volume measure $\mu$ in this particular case using arguments restriced to the case of the GFF such as the domain Markov property and obtained a version of the KPZ relations.
 Simultaneously, Rhodes and Vargas, \cite{RhodesVargas}, among other things, showed that Gaussian multiplicative chaos can be used directly to construct the same object. They also gave a simpler and  more general proof of a stronger version of the KPZ relation. We refer the reader to \cite{Garban} for an excellent introduction to this area. See also \cite{GFFnotes} for a more detailed exposition.

Kahane's original work assumes that the covariance kernel $K$ of $h$ is $\sigma$-positive, meaning that $K(x,y)$ may be written as the pointwise sum $K(x,y) = \sum_k K_k(x,y)$ where the summands $K_k$ are nonnegative symmetric definite continuous functions and, crucially, $K_k(x,y) \ge 0$ pointwise. Under this assumption (which is somewhat restrictive as it is hard to check in practice), he was able to show that a truncation of $h$ associated with the $\sigma$-positive decomposition of $K$ gives rise to a well-defined measure $\mu$ as in \eqref{goal} and characterised the values of $\gamma$ for which it is nontrivial for a given reference measure $\sigma$. He also studied fine properties of the resulting random measure $\mu$ and showed that its law does not depend on the decomposition of the $\sigma$-positive kernel $K$ into  positive summands.

Much more recently, Robert and Vargas \cite{RobertVargas} (motivated by applications to three-dimensional turbulence) obtained a significant generalisation of this theory. They were able to show that, without assuming $\sigma$-positivity, regularising the field with a general mollifier function $\theta$ subject to mild assumptions, gives rise to a sequence of measures $\mu_\eps$ such that $\mu_\eps(S)$ converges in law and the law of the limit does not depend on the regularising function $\theta$.  Even more recently, Shamov \cite{shamov} showed in a very general setting that convergence holds in probability and the limit does not depend on the regularisation. In particular the measure $\mu$ is measurable with respect to $h$. (Conversely,  by a result of \cite{BSS}, $h$ is measurable with respect to $\mu$, at least in the case of the two-dimensional GFF and $\sigma(dz) = dz$ being the Lebesgue measure). This was also the subject of a recent preprint by Junnila and Saksman \cite{JunnilaSaksman} whose results, remarkably, cover the critical case.

\medskip The purpose of this short note is to provide an elementary and completely self-contained proof of Kahane's theory together with some of the important developments above. Eventually we are able to reprove convergence in probability (and in $L^1$) and show nontriviality in the entire subcritical phase $(\gamma < \sqrt{2d}$), together with the universality result showing uniqueness of the limit (independence with respect to the regularisation function $\theta$). While the setup is slightly less general than Shamov \cite{shamov}, we feel that the result and its proof are nevertheless interesting because of the completely elementary nature of the arguments, and the fact that they cover the most interesting cases without significant assumptions on the covariance kernel $K$ (in particular, no $\sigma$-positivity assumption is made).

\medskip \noindent \textbf{Assumptions: } Let $D \subset \R^k$ be a domain.
Consider a nonnegative definite kernel $K(x,y)$ of the form
\begin{equation}\label{cov}
 K(x,y) = \log (|x-y|^{-1} ) + g(x,y)
\end{equation}
where $g$ is continuous over $\bar D \times \bar D$. Set
$$\cM_+ =\{ \rho \text{ nonnegative measure in $D$ such that:} \iint |K(x,y)| \rho(dx) \rho(dy) < \infty\},$$ and set $\cM$ be the set of signed measures of the form  $ \rho = \rho_+ - \rho_-$, where $\rho_\pm \in \cM_+$, and note that $\cM$ contains all smooth compactly supported functions in $D$. Let $h$ be the centered Gaussian generalised function with covariance $K$. That is, we view $h$ as a stochastic process indexed by $\cM$, characterised by the two properties that: $(h, \rho)$ is linear in $\rho \in \cM$ in the sense that $(h, \alpha \rho_1 + \beta \rho_2 ) = \alpha (h, \rho_1) + \beta(h, \rho_2)$ almost surely; and for any $\rho \in \cM$,
$$
(h,\rho) \text{ is a centered Gaussian random variable with variance }\iint  K(x,y) \rho(dx)\rho(dy).
$$
We will write $\int h(x) \rho(dx) $ for the random variable $(h,\rho)$ with an abuse of notation. Note that this setup covers the case of a Gaussian free field in two dimensions with say Dirichlet boundary conditions (but also the case of free or Neumann boundary conditions, by changing if necessary $\gamma$ into $2\gamma$). See \cite{GFFnotes} (which takes a similar viewpoint) for more details on this. We extend the definition of $h$ outside of $D$ by setting $h|_{D^c} = 0$, so for any measure $\rho$ such that $\rho|_D \in \cM$, by definition $(h, \rho) = (h , \rho|_D)$.

Let $\sigma$ be a Radon measure on $\bar  D$ of dimension at least $d$ (where $0\le d \le k$), i.e.,
\begin{equation}\label{dim}
\iint_{\ \bar D\times \bar D} \frac1{| x-y|^{d-\eps}}\sigma(dx) \sigma(dy) < \infty
\end{equation}
for all $\eps>0$. In particular $\sigma$ is a finite measure since $\bar D$ is bounded.
Let $S \subset D$ and let $\theta$ be a fixed nonnegative Radon measure on $\R^k$ supported in the unit ball $B(0,1)$, such that $\theta(\R^d) = 1$ and
\begin{equation}\label{eq:assumptiontheta}
\int |\log (1/|x-y|)| \theta(dy) \le C < \infty
\end{equation}
where $C$ does not depend on $x \in B(0,5)$. It is easy to check that the condition \eqref{eq:assumptiontheta} is satisfied whenever $\theta$ has a Lebesgue density in $L^p$ for some $p>1$ supported in $B(0,1)$, and also in many other cases, e.g. the uniform distribution on the unit circle.
Set $\theta_\eps(\cdot)$ to be the image of the measure $\theta$ under $x \mapsto \eps x$, i.e. $\theta_\eps(A) =  \theta (A/\eps)$ for all Borel sets $A$, which we view as an approximation of the identity based on $\theta$ (we will sometimes write $\theta_\eps(x)dx$ for the measure $\theta_\eps(dx)$ with an abuse of notations). We will also write $\theta_{x, \eps}(\cdot)$ for the measure $\theta_\eps$ translated by $x$. For $x \in S$, note that\footnote{It is tempting to only use $\theta \in \cM$ instead of \eqref{eq:assumptiontheta} as an assumption on $\theta$. However this leads to problems in the proofs below, as pointed out by an anonymous referee, even if we assume $\theta$ has a density. Consider the following instructive example in dimension $d=1$. Take $f(x) = (c/|x|) (1+ \log^+(1/|x|))^{-2}$ for $x \in (-1,1)$ and $f(x) = 0$ else. Then if $\theta(dx) = f(x) dx$, we have $\theta \in \cM$. However, it is not the case that \eqref{eq:roughcov} holds. In fact even the basic fact \eqref{covbound} does not hold, as the left hand side is unbounded as $t \to \infty$ (since $\int \log(1/|x|) f(x) dx = \infty$).\label{footnote}}
 by \eqref{eq:assumptiontheta}, the translated measure $\theta_{x,\eps}\in \cM$, so we can define an $\eps$-regularisation of the field $h$ by setting for $\eps$ small,
\begin{equation}\label{convol}
h_\eps(x) = h \ast \theta_\eps(x) = \int h(y) \theta_\eps( x-y ) dy  = \int h(y) \theta_{x, \eps}(dy), x \in S.
\end{equation}
One can check that $\var (h_\eps(x) - h_\eps(x')) \to 0 $ as $|x-x'| \to0$ for a fixed $\eps$, so there exists a version of the stochastic process $h$ such that $h_\eps(x)$ is almost surely a Borel measurable function of $x \in S$ (see e.g. Proposition 2.1.12 in \cite{GN}).
Hence let
$$
\mu_\eps(S) = I_\eps = \int_S e^{\gamma h_\eps(z) - \frac{\gamma^2}{2}  \E( h_\eps(z)^2 ) }\sigma(dz).
$$
In the case of 2d Liouville quantum gravity, the natural choice for $\sigma$ is often $\sigma(dz) = R(z, D)^{\gamma^2/2}dz$ where $R(z,D)$ denotes the conformal radius of the point $z$ in $D$, and the natural choice for the measure $\theta$ is often the uniform distribution on the unit circle (so $h_\eps(z)$ is the usual circle average process of $h$). However, the case where $\sigma$ is the occupation measure of an (independent) planar Brownian motion is also of interest as it is used for defining the Liouville Brownian motion \cite{GRV, NB}, the canonical diffusion process in Liouville quantum gravity. In this example, $\sigma$ is singular with respect to Lebesgue measure, yet $\sigma$ is of dimension two in the sense of \eqref{dim}.

\begin{thm}\label{T:conv}
Assume that $\gamma < \sqrt{2d}$. Then $\mu_\eps(S)$ converges in probability and in $L^1(\P)$ to a limit $\mu(S)$. Furthermore the random variable $\mu(S)$ does not depend on the choice of the regularising kernel $\theta$ subject to the above assumptions. Furthermore, $\mu$ defines a Borel measure on $D$ and $\mu_\eps$ converges in probability towards $\mu$ for the topology of weak convergence of measures on $D$.
\end{thm}


\noindent \textbf{Acknowledgements:} The idea for this work emerged while preparing a presentation for a reading group at the Newton institute during the semester on \emph{Random Geometry}. I thank the participants of the reading group for their questions. I am particularly grateful to Juhan Aru and St\'ephane Beno\^ist for some useful discussions and comments on a preliminary draft. Thanks also to Ofer Zeitouni for some comments on the Karhunen--Lo\` eve expansion of Gaussian fields, to Pascal Maillard for pointing out some typos in an early version of the paper, and to Ellen Powell for useful comments while preparing the final version. An anonymous referee provided very useful suggestions which improved the presentation of the paper and made a number of interesting observations, some of which I chose to include in footnote \ref{footnote}; I am very grateful for this.



\section{Main idea}

It is well known and relatively easy to see that for $\gamma$ sufficiently small (namely $\gamma < \sqrt{d}$), the multiplicative measure $\mu_\eps$ are uniformly integrable: indeed the quantity $\mu_\eps(S)$ is then bounded in $L^2$, hence any limit must be nontrivial.

Therefore difficulties mainly arise in the phase where $\gamma \in [\sqrt{d}, \sqrt{2d})$. (In Liouville Quantum Gravity, this is the phase of principal interest as this is precisely the measures which are thought to arise as scaling limits of FK-weighted planar maps).  The main idea for this work is the following very elementary observation. Any limiting measure $\mu$ must be supported on the so-called \emph{$\gamma$-thick points} of the field $h$: that is, on points $x$ such that
\begin{equation}
\lim_{\eps\to 0} \frac{h_\eps(x) }{\log (1/\eps)} = \gamma.
\end{equation}
Such points were studied in detail in the case of the two-dimensional Gaussian Free Field by \cite{HMP} but a related notion was already apparent in the early work of Kahane \cite{Kahane} who pointed to its importance. That any limiting measure would have to be supported by $\gamma$-thick points is apparent from the definition of $\mu_\eps$ and Girsanov's lemma (or rather the Cameron--Martin formula). Indeed this implies that, when biasing the law of the field by a factor proportional to $e^{\gamma h_\eps (x) }$, the mean value of $h_\eps(x)$ is shifted from 0 to $\gamma \var( h_\eps(x)) = \gamma \log(1/\eps) + O(1)$.

Therefore, one can pick $\alpha > \gamma$ and call a point $x$ bad if its thickness is greater than $\alpha$, and good otherwise (in fact the key will be to take a slightly more restrictive definition of good points). We then consider the normalised measure $e^{\gamma h_\eps(x)}dx$, but \emph{restricted to good points}. As it turns out, the $L^1$ contribution of bad points is easily shown to be negligible (essentially by the above Cameron--Martin--Girsanov observation), while the remaining part is shown to remain bounded in $L^2(\P)$. We will see that our definition of \emph{good points} allows one to make the relevant $L^2$ computation very simple.

Convergence is then shown to be a consequence of the $L^2$ boundedness of the good part (roughly, the good part is a Cauchy sequence in $L^2(\P)$), while the bad part is small in $L^1(\P)$. Uniqueness comes from the fact that once uniform integrability is established for the regularisations of the field such as \eqref{convol}, we can also get uniform integrability for another approximation of the field, this time arising from the Karhunen--Loeve expansion of $h$. This gives another approximation of the measure which turns out to be a martingale, and hence also has a limit. We then show that the two measures must agree, thereby deducing uniqueness.

\section{Uniform Integrability}
The goal of this section will be to prove:

\begin{prop}\label{P:UI}
$I_\eps$ is uniformly integrable.
\end{prop}

\begin{proof}
Let $\alpha > 0$ be fixed (it will be chosen $> \gamma$ and very close to $\gamma$ soon). We will use the following notation in the rest of the article: for $r >0$ we define
\begin{equation}\label{rbar}
\bar r = e^{  \lceil \log r \rceil} = \inf \{ e^k: k \in \Z, e^k > r \}
\end{equation}
 to be the closest upper $e$-adic approximation of $r$. We define a good event
$$
G^\alpha_\eps(x) = \{ h_{\bar r}(x) \le \alpha \log (1/\bar r) \text{ for all $r \in [\eps, \eps_0] $} \}
$$
with $\eps_0 \le 1$ for instance. This is the good event that the point $x$ is never too thick up to scale $\eps$.
 Further let $\bar h_\eps(x) = \gamma h_\eps(x) - (\gamma^2/2 )  \E( h_\eps(x)^2)$ to ease notations.

\begin{lemma}[Ordinary points are not thick]
\label{L:typicalthick} For any $\alpha >0$, we have that uniformly over $x\in S$, $\P(G^\alpha_\eps(x)) \ge 1- p(\eps_0)$ where the function $p$ may depend on $\alpha$ and for a fixed $\alpha > \gamma$, $p(\eps_0) \to 0$ as $\eps_0 \to 0$. \end{lemma}

\begin{remark}
The proof of this lemma is trivial in the case of the two-dimensional GFF (in particular no general machinery about Gaussian processes is needed in this case).
\end{remark}

\begin{proof}
Set $X_t = h_\eps(x)$ for $\eps = e^{- t}$. Then a direct computation from \eqref{cov} (see below in Lemma \ref{L:cov}, and more precisely \eqref{eq:roughcov}), implies that
\begin{equation}\label{covbound}
|\cov(X_s, X_t) - s\wedge t| \le O(1),
\end{equation}
 where the implicit constant is uniform.
In particular $\var (X_t ) = t + O(1)$.

Note that for each $k\ge 1$, $\P( X_k \ge \alpha k /2) \le e^{- \alpha^2 k^2/ ( 8\var (X_k))} $ which decays exponentially in $k$ by the above, and so is smaller than $Ce^{- \lambda k}$ for some $\lambda>0$.
Hence
$$
\P( \exists k \ge k_0:|X_k | \ge \alpha k) \le \sum_{k \ge k_0} Ce^{- \lambda k }
$$
We call $p(\eps_0) $ to be the right hand side of the above for $k_0 = \lceil-\log(1/\eps_0)\rceil $ which can be made arbitrarily small by picking $\eps_0$ small enough.
This proves the lemma.
\end{proof}

\begin{lemma}[Liouville points are no more than $\gamma$-thick]  \label{L:Liouvillethick} For $\alpha > \gamma$ we have
$$
\E(e^{\bar h_\eps(x) } 1_{G^\alpha_\eps(x)} ) \ge 1- p(\eps_0).
$$
\end{lemma}

\begin{proof}
Note that
$$
\E( e^{ \bar h_\eps(x)}  \indic{G^\alpha_\eps(x) }) = \tilde \P (  G^\alpha_\eps(x)), \text{ where } \frac{d\tilde \P}{d\P} =  e^{\bar h_\eps(x)}.
$$
By the Cameron--Martin--Girsanov lemma, under $\tilde \P$, the process $(X_s)_{- \log \eps_0 \le s\le t }$ has the same covariance structure as under $\P$ and its mean is now $\gamma \cov(X_s, X_t) = \gamma s + O(1) $ for $ s\le t$. Hence
$$
\tilde \P( G^\alpha_\eps(x))  \ge \P (  G^{\alpha - \gamma }_\eps(x) ) \ge 1-p(\eps_0)
$$
by Lemma \ref{L:typicalthick} since $\alpha > \gamma$.
\end{proof}

We therefore see that points which are more than $\gamma$-thick do not contribute significantly to $I_\eps$ in expectation and can therefore be safely removed. We therefore fix $\alpha >\gamma$ and introduce:
\begin{equation}
J_\eps = \int_S  e^{ \bar h_\eps(z)} \indic{G_\eps(z)}\sigma(dz)
\end{equation}
with $G_\eps(x) = G^\alpha_\eps(x)$. We will show that $J_\eps$ is uniformly integrable from which the result follows.

Before we embark on the main argument of the proof, we record here for ease of reference an elementary estimate on the covariance structure of $h_\eps(x)$. Roughly speaking, the role of the first estimate \eqref{eq:roughcov} is to bound from above (up to an unimportant constant of the form $e^{O(1)}$) the contribution to $\E(J_\eps^2)$ coming from points $x,y$ that are close to each other. That will suffice to prove uniform integrability. The role of the finer estimate \eqref{eq:finecov} is to get a more precise estimate to the contribution to $\E( J_\eps^2)$ coming from points $x,y$ which are macroscopically far away, which we will be able to assume thanks to \eqref{eq:roughcov}. This time the error in the covariance up to an additive term $o(1)$ will translate into an error up to a factor $e^{o(1)} = 1+ o(1)$ in the estimation of this contribution. In turn this will imply convergence.

\begin{lemma}
  \label{L:cov} We have the following estimate:
  \begin{equation}\label{eq:roughcov}
  \cov( h_\eps(x), h_r(y)) = \log 1/(|x-y|\vee r \vee \eps) + O(1).
  \end{equation}
  Moreover, if $\eta>0$ and $|x-y| \ge \eta$, then
  \begin{equation}\label{eq:finecov}
  \cov( h_\eps(x), h_\delta(y)) = \log (1/|x-y|) + g(x,y) + o(1)
  \end{equation}
  where $o(1)$ tends to 0 as $\delta, \eps \to 0$, uniformly in $|x- y | \ge \eta$.
\end{lemma}

\begin{proof}
  We start with the proof of \eqref{eq:roughcov}. Assume without loss of generality that $\eps \le r$. Note that
  \begin{align}
  \cov( h_\eps(x), h_r(y)) &= \iint K(z,w) \theta_{x, \eps}(dw) \theta_{y,r}(dz)\nonumber\\
  & = \iint -\log (| w-z |)  \theta_{x,\eps}(dw) \theta_{y,r}(dz) + O(1)\label{roughcov1}
  \end{align}
  We consider the following cases: (a) $ r \le |x-y |/3$, and (b) $r \ge | x-y | /3$.

  In case (a),  $|x-y| \le \eps + |w-z| +r \le 2 r + |w-z| \le (2/3) |x-y| + |w-z|$ by the triangle inequality, so $|w-z| \ge (1/3) | x-y|$ and we get
   $$
    \cov( h_\eps(x), h_r(y)) \le - \log |x-y| + O(1)
   $$
    as desired in this case.

  The second case (b) is when $r \ge |x-y|/3$. Then by translation and scaling so that $B(y,r)$ becomes $B(0,1)$, the right hand side of \eqref{roughcov1} is equal to
  $$
  \log (1/r) + \iint -\log |w-z| \theta_{\frac{x-y}{r}, \frac{\eps}{r}} (dw) \theta(dz)
  $$
Conditioning on $w$ (which is necessarily in $\bar B(0,4)$ under the assumptions of case (b)), we see that by the assumption \eqref{eq:assumptiontheta} on $\theta$, the second term is bounded by $O(1)$, uniformly, so that
$$
    \cov( h_\eps(x), h_r(y)) \le - \log r + O(1)
   $$
   as desired in this case. This proves \eqref{eq:roughcov}.

  The proof of \eqref{eq:finecov} is similar but simpler. Indeed, we get (as in \eqref{roughcov1}),
  \begin{equation}\label{finecov1}
  \cov(h_\eps(x), h_\delta(y)) = \iint - \log | w-z| \theta_{x,\eps}(dw) \theta_{y, \delta}(dz) + g(x,y) + o(1)
  \end{equation}
  where the $o(1)$ term tends to 0 as $\eps, \delta \to 0$, coming from the continuity of $g$, and hence is uniform in $x,y$ (not even assuming $|x-y | \ge \eta$). Now note that
  $$
  \big| \log | w- z | - \log |x- y | \big| \le  \frac{4 \max(\eps, \delta)}{|x-y|}
  $$
  as soon as $\max(\eps, \delta ) \le \eta/4 \le |x-y |/4$. Therefore the right hand side of \eqref{finecov1} is $- \log | x- y |  +  g(x,y) + O( \max (\eps, \delta)) + o(1) $ when $|x-y | \ge \eta$, which proves the claim \eqref{eq:finecov}.
\end{proof}


\begin{lemma}
\label{L:UI}
For $\alpha>\gamma$ sufficiently close to $\gamma$, $J_\eps$ is bounded in $L^2(\P)$ and hence uniformly integrable.
\end{lemma}

\begin{proof} By Fubini's theorem,
\begin{align*}
\E (J_\eps^2) & = \int_{S\times S} \E( e^{\bar h_\eps(x) + \bar h_\eps(y)}\indic{G_\eps(x) \cap G_\eps(y)})  \sigma(dx )\sigma(dy)\\
& = \int_{S\times S}   e^{\gamma^2 \cov (h_\eps(x), h_\eps(y))}
\tilde \P ( G_\eps(x)\cap  G_\eps(y)) \sigma(dx)\sigma( dy)
\end{align*}
where $\tilde \P$ is a new probability measure obtained by the Radon-Nikodyn derivative
$$
\frac{d\tilde \P}{d\P} = \frac{e^{\bar  h_\eps(x) + \bar h_\eps(y)}}{\E(e^{\bar  h_\eps(x) + \bar h_\eps (y)}) }.
$$
By Lemma \ref{L:cov} (more precisely by \eqref{eq:roughcov})
\begin{equation}\label{twopointcorr}
\cov(h_\eps(x), h_\eps(y) ) =  - \log(|x-y| \vee \eps) + g(x,y) + O(1).
\end{equation}
Also, if and $\eps \le e^{-1} \eps_0$ and $|x- y | \le e^{-1} \eps_0  $ (else we bound the probability below by one), we have
$$
\tilde \P ( G_\eps(x) \cap  G_\eps(y)) \le \tilde \P ( h_{r} (x) \le \alpha \log 1/r)
$$
where
\begin{equation}\label{eps'}
r = \overline{\eps \vee |x-y|}
\end{equation}
(recall our notation for $\bar r = \inf\{ e^k: k \in \Z, e^k > r\}$, see \eqref{rbar}.)
Furthermore, by Cameron--Martin--Girsanov, under $\tilde \P$ we have that $h_r(x)$ has the same variance as before (therefore $\log 1/r + O(1)$) and a mean given by
\begin{equation}\label{twopointcorr2}
\cov_{\P} ( h_{r}(x), \gamma h_\eps(x) + \gamma h_\eps(y)) = 2\gamma \log 1/r + O(1),
\end{equation}
again by Lemma \ref{L:cov} (more precisely, by \eqref{eq:roughcov}).
Consequently,
\begin{align}
 \tilde \P ( h_{r} (x) \le \alpha \log 1/r) & = \P( \cN( 2\gamma \log (1/r), \log 1/r) \le \alpha \log (1/r) + O(1)) \nonumber \\
& \le \exp ( - \frac12 (2\gamma - \alpha)^2 (\log (1/r) + O(1))) = O(1) r^{(2\gamma - \alpha)^2/2}.\label{corr_naive}
\end{align}
We deduce
\begin{align}
\E(J_\eps^2) & \le O(1) \int_{S\times S} |(x-y )\vee \eps|^{ ( 2\gamma - \alpha)^2/2-\gamma^2}  \sigma(dx )\sigma(dy).\label{J2naive}
\end{align}
(We will get a better approximation in the next section).
Clearly by \eqref{dim} this is bounded if
$$
( 2\gamma - \alpha)^2/2-\gamma^2 > - d
$$
and since $\alpha$ can be chosen arbitrarily close to $\gamma$ this is possible if
\begin{equation}\label{UI}
d - \gamma^2 /2 >0 \text{ or } \gamma < \sqrt{2d}.
\end{equation}
This proves the lemma.
\end{proof}

To finish the proof of Proposition \ref{P:UI}, observe that $I_\eps = J_\eps + J'_\eps$.
We have $\E(J'_\eps) \le p(\eps_0)$ by Lemma \ref{L:Liouvillethick},  and for a fixed $\eps_0$, $J_\eps$ is bounded in $L^2$ (uniformly in $\eps$). Hence $I_\eps$ is uniformly integrable.
\end{proof}

\section{Convergence}

As before, since $\E(J'_\eps)$ can be made arbitrarily small by choosing $\eps_0$ sufficiently small, it suffices to show that $J_\eps$ converges in probability and in $L^1$. In fact we will show that it converges in $L^2$, from which convergence will follow. To do this we will show that $(J_\eps)$ forms a Cauchy sequence in $L^2$, and we start by writing
\begin{equation}\label{basic}
\E( (J_\eps - J_\delta)^2 ) = \E(J_\eps^2 ) + \E ( J_\delta^2) - 2 \E(J_\eps J_\delta)
\end{equation}

Our basic approach is thus to estimate better than before $\E(J_\eps^2)$ from above and $\E(J_\eps J_\delta)$ from below. Essentially, the idea is that for $x,y$ which are at a small but macroscopic distance, we can identify the limiting distribution of $(h_{r}(x), h_{r}(y))_{r\le \eps_0}$ under the distribution $\P$ biased by $e^{\bar h_\eps(x) + \bar h_\delta(y)}$. On the other hand when $x,y$ are closer than that we know from the previous section that the contribution is essentially negligible.

\begin{lemma}
\label{L:limsup}
We have
$$
\limsup_{\eps\to 0} \E( J_\eps^2) \le \int_{S\times S} e^{\gamma^2 g(x,y)} \frac1{|x-y|^{\gamma^2}}  g_\alpha(x,y)\sigma(dx)\sigma( dy)
$$
where $g_\alpha(x,y)$ is a nonnegative function depending on $\alpha, \eps_0$ and $\gamma$ such that the above integral is finite.
\end{lemma}
\begin{proof}
Recall that from \eqref{twopointcorr} we already know
$$
\E(J_\eps^2) = \int_{S^2}  e^{\gamma^2  \cov( h_\eps(x), h_\eps(y))} \tilde \P( G_\eps(x)\cap G_\eps(y)) \sigma(dx)\sigma( dy).
$$
We simply have to estimate better $\tilde \P( G_\eps(x)\cap  G_\eps(y)  ) $. We fix $\eta >0$ arbitrarily small (in particular, $\eta$ may and will be smaller than $e^{-1}\eps_0$). If $|x-y| \le \eta$ we use the same bound as in \eqref{J2naive}. The contribution coming from the part $|x-y| \le \eta$ can thus be bounded, uniformly in $\eps$, by $f(\eta) $ (where $f(\eta) \to 0$ as $\eta \to 0$ and the precise order of magnitude of $f(\eta)$ is determined by \eqref{dim}, and is at most polynomial in $\eta$). We thus focus on the contribution coming from $|x- y| \ge \eta$.

Then observe that for any fixed $\eps_1 \le \eps_0$, as $\eps \to 0$, and uniformly over $x \in S$ and $r \ge \eps_1$,
\begin{equation}\label{shift1}
\cov( h_{r}(x), h_\eps(x) ) \to \int_D K(x,z) \theta_{r}(x-z) dz
\end{equation}
and likewise, uniformly over $x, y \in S$ such that $|x-y|\ge \eta$, and over $r \ge \eps_1$, as $\eps \to 0$:
\begin{equation}\label{shift2}
\cov(h_{r}(x), h_\eps(y)) \to \int_D K(z,y) \theta_{r}(x-z) dz
\end{equation}
(Note that both right hand sides of \eqref{shift1} and \eqref{shift2} are finite by \eqref{eq:roughcov}.) Consequently, by Cameron--Martin--Girsanov, the joint law of the processes $(h_{r}(x), h_{r}(y))_{r\le \eps_0}$ under $\tilde \P$ converges to a joint distribution
$(\tilde h_{r}(x), \tilde h_{r}(y))_{r \le \eps_0}$ whose covariance is unchanged and whose mean is given by the sum of \eqref{shift1} and \eqref{shift2} times $\gamma$. This convergence is for the weak convergence on compacts of $r \in (0, \eps_0]$, and is uniformly in $|x-y| \ge \eta$.
Let $\tilde G(x) $ be the event that $\tilde h_{r}(x)  \le \alpha \log (1/r)$ for all $r \le \eps_0$. Then it is not hard to deduce, uniformly in $|x- y | \ge \eta$,
\begin{equation}\label{corr_better}
\tilde \P( G_\eps(x) \cap G_\eps(y)) \to g_\alpha(x,y): = \P( \tilde G (x) \cap  \tilde G(y))\ \  (\eps \to 0) .
\end{equation}
Indeed, by \eqref{eq:roughcov}, under $\tilde \P$, the drifts of $h_r(x)$ and of $h_r(y)$ (with $ r \ge \eps$) are each $\gamma \log (1/r) + O(1)$ where the $O(1)$ term is uniform in $|x-y | \ge \eta$. Because of this, up to an error in $\tilde \P$ probability that is arbitrarily small (uniformly in $|x-y | \ge \eta$), the events $G_\eps(x), G_\eps(y)$ as well as $\tilde G(x), \tilde G(y)$ depend only on the ``macroscopic" behaviour of $h_r(x)$ and $h_r(y)$; that is, depend only on $(h_r(x), h_r(y))_{r \ge \eps_1}$ for some $\eps_1$.
Consequently, as $\eps \to 0$, after applying Lemma \ref{L:cov} (and more specifically \eqref{eq:finecov}), we deduce (using \eqref{J2naive} to justify the use of dominated convergence):
\begin{equation}\label{convJ2}
\int_{S^2; |x- y|\ge \eta} \!\!\!\!\!\!\!\!\!\!\!\!\!\!\!\!{e^{\gamma^2 \cov( h_\eps(x), h_\eps(y))}} \tilde \P( G_\eps(x), G_\eps(y))\sigma( dx ) \sigma( dy )\to \int_{S^2 ; |x- y | \ge \eta} \frac{e^{ \gamma^2 g(x,y)}}{|x-y|^{\gamma^2}} g_\alpha (x,y) \sigma(dx) \sigma(dy).
\end{equation}
Since we already know that the piece of the integral coming from $|x- y | \le \eta$ contributes at most $f(\eta) \to 0$ when $\eta \to 0$, it remains to check that the integral on the right hand side of \eqref{convJ2} remains finite as $\eta \to 0$. But we have already seen in \eqref{corr_naive} that for $|x- y | \le \eps_0/3$, $\tilde \P( G_\eps(x) \cap  G_\eps(y)) \le O(1) |x-y|^{( 2\gamma - \alpha)^2/2-\gamma^2}$; hence this inequality must also hold for $g_\alpha (x,y)$. Hence the result follows as in \eqref{UI}.
\end{proof}

\begin{lemma}
\label{L:liminf}
We have
$$
\liminf_{\eps,\delta \to 0} \E( J_\eps J_\delta ) \ge \int_{S\times S} e^{\gamma^2 g(x,y)} \frac1{|x-y|^{\gamma^2}}  g_\alpha(x,y) \sigma(dx) \sigma( dy).
$$
\end{lemma}

\begin{proof}
In fact, the proof is almost exactly the same as in Lemma \ref{L:limsup}, except that
$\tilde \P$ is now weighted by $e^{\bar h_\eps(x) + \bar h_\delta(y)}$ instead of $e^{\bar h_\eps(x) + \bar h_\eps(y)}$.  But this changes nothing to the argument leading up to \eqref{corr_better} and hence \eqref{convJ2} still holds. Since we get a lower bound by restricting ourselves to $|x- y |  \ge \eta$, we deduce immediately that
$$
\liminf_{\eps,\delta \to 0} \E( J_\eps J_\delta ) \ge \int_{S^2; | x- y  | \ge \eta} e^{\gamma^2 g(x,y)} \frac1{|x-y|^{\gamma^2}}  g_\alpha(x,y) \sigma(dx) \sigma( dy).
$$
Since $\eta$ is arbitrary, the result follows.
\end{proof}
\begin{proof}[Proof of convergence in Theorem \ref{T:conv}]
Using \eqref{basic} together with Lemmas \ref{L:limsup}, \ref{L:liminf}, we see that $J_\eps$ is a Cauchy sequence in $L^2$ for any $\eps_0>0$.
Combining with Lemma \ref{L:Liouvillethick}, it therefore follows that $I_\eps$ is a Cauchy sequence in $L^1$ and hence converges in $L^1$ (and also in probability) to a limit $I = \mu(S)$.
 \end{proof}

\begin{remark}
Note that $\lim_{\eps\to 0} \E(J_\eps^2)$ depends on the regularisation $\theta$,
even though, as we will see next, $\lim_{\eps \to 0} I_\eps$ does not.
\end{remark}

\section{Uniqueness of the limit}

For the proof of independence of the limit with respect to the regularising kernel $\theta$ we may assume without loss of generality that $D$ is bounded.

\begin{lemma}
\label{L:expansion}
We may write
$
h= \sum_{n=0}^\infty h_n
$
where the $h_n$ are independent continuous Gaussian fields, in the sense that for an arbitrary fixed function $f \in L^2(D,dx)$, $\sum_{n=0}^\infty (h_n, f) $ converges almost surely and the limit agrees with $(h,f)$ almost surely.
\end{lemma}

\begin{remark}
  In the case of the Gaussian free field, we can write $h = \sum_i X_i f_i$, where $f_i$ is an orthonormal basis of the Sobolev space $H_0^1(D)$, and $X_i$ are i.i.d. standard normal random variables, hence we can take $h_n = X_n f_n$ in this case.
\end{remark}

\begin{proof}
This is basically the Karhunen-Loeve expansion of $h$ (see \cite{oz}). Since $h$ is only a stochastic process indexed by $\cM$ we do it carefully. Introduce the Fredholm integral operator
$$
Tf(x) = \int_D K(x,y)f(y) dy,\ \  f \in L^2(D;dy).
$$
Note that $T$ is well defined on $L^2(D,dy)$ and maps $L^2(D,dy)$ into continuous functions on $\bar D$ by Lebesgue's dominated convergence theorem and Cauchy--Schwarz. Since $D$ is bounded, we deduce that $T:L^2(D,dy) \to L^2(D,dy)$. Note further that since $K(x,y) = K(y,x)$ by assumption, $T$ is symmetric with respect to the (Lebesgue) inner product on $L^2(D)$. Observe also that since $K(x,y) \in L^2(D\times D,dxdy)$, we have that $T$ is a compact operator on $L^2(D,dy)$ (follows from equicontinuity and Arzela--Ascoli). By the spectral theorem for compact symmetric operators, we deduce that there exists an orthonormal basis of $L^2(D,dy)$ consisting of eigenfunctions $\{f_k\}_{k\ge 0}$ of $T$ (Theorem 7 in Appendix D.5 of \cite{Evans}). Let $\lambda_k$ denote the corresponding eigenvalue. We have that $\lambda_k > 0$ and $\lambda_k \to 0$ as $k \to \infty$ by the same theorem. Observe that trivially $f_k$ must be continuous since $Tf_k = \lambda_k f_k$ and $f_k \in L^2$ so $Tf_k$ is continuous.

Now consider our field $h$. Observe that $\xi_k = (h,f_k) / \sqrt{\lambda_k}$ is well defined (by Cauchy--Schwarz, since $f_k \in L^2(D)$) and forms an i.i.d. sequence of standard Gaussian random variables:
\begin{align*}
\E( \xi_k \xi_j)& = \frac1{\sqrt{\lambda_k \lambda j}} \iint K(x,y) f_k(x) f_j(y)dx dy = \frac1{\sqrt{\lambda_k \lambda j}}  \int f_k(x) Tf_j(x)dx  \\
& = \frac{\lambda_j}{\sqrt{\lambda_k \lambda_j}} (f_k, f_j) = \indic{j=k}.
\end{align*}

Set $h_k(x): =\sqrt{\lambda_k}  \xi_k f_k(x)$. Note that $h_k$ is then a.s. continuous and in $L^2(D;dx)$. Observe that for an arbitrary test function $f \in L^2(D,dy)$, the sequence $M_n(f) = \sum_{k=0}^n (h_k, f)$ defines a martingale in the filtration generated by $(\xi_1, \ldots)$. Note that by independence of the $(\xi_k)_{k\ge 0}$, and by Parseval's identity, as $n \to \infty$,
\begin{align*}
\var M_n(f) = \sum_{k=0}^n \lambda_k (f, f_k)^2  & \to \sum_{k=0}^\infty \lambda_k (f,f_k) \cdot (f,f_k) = \sum_{k=0}^\infty (Tf, f_k) (f, f_k) = (Tf, f) \\
& = \var (h,f) < \infty
\end{align*}
Hence $M_n(f) $ converges a.s. and in $L^2(\P)$. Moreover the same calculation shows that in fact
$\var [(h,f) - \sum_{k=0}^n (h_k ,f)]$
converges to 0: indeed, as $n\to \infty$,
\begin{align*}
  \cov [(h,f) ; \sum_{k=0}^n (h_k ,f)] & = \sum_{k=0}^n \cov \Big((h,k) ; \int_D h_k(x) f(x) dx\Big)\\
  &
 = \sum_{k=0}^n \int_D f(x) \cov \Big((h, f); (h, f_k) f_k(x)\Big) dx\\
  & = \sum_{k=0}^n \int_D f(x) f_k(x) (f, Tf_k) dx = \sum_{k=0}^n \lambda_k (f, f_k)^2 \to \var (h,f).
\end{align*}
 Hence the almost sure limit of $M_n(f)$ agrees with $(h,f)$, as desired.
\end{proof}

Now define $h^n (z) = \sum_{k=0}^n h_k(z)$ and set
$$
\mu^n(S) = \int_S\exp\left( \gamma h^n(z) - \frac{\gamma^2}{2} \var (h^n(z)) \right) \sigma(dz).
$$
Then $\mu^n(S)$ is a positive martingale with respect to the filtration $\cF_n = \sigma(\xi_1, \ldots, \xi_n)$ so converges to a limit which we will call $\mu'(S)$.

\begin{proof}[Proof of Theorem \ref{T:conv}; uniqueness.] It suffices to show that $\mu(S) = \mu'(S)$. This will show that $\mu(S)$ does not depend on the regularisation kernel. Observe that
$$
\E( \mu_\eps(S) | \cF_n) = \mu_\eps^n(S)
$$
where $\mu^n_\eps$ is defined as $\mu^n$ except with $h^n$ replaced by its regularisation $h^n_\eps = h^n \ast \theta_\eps$. [This follows from writing $h = h^n + X$ where $X$ is independent from $h^n$.]
When $n$ is finite and $\eps\to 0$ there is no problem in showing that the right hand side converges almost surely to $\mu^n (S)$, by continuity of $h^n$.
Hence the left hand side also converges a.s. to some limit as $\eps \to 0$, and we have
\begin{equation}\label{neps}
\mu^n(S) = \lim_{\eps \to 0} \E( \mu_\eps(S) | \cF_n).
\end{equation}
However, we have shown that $\mu_\eps(S)$ converges in $L^1(\P)$ to $\mu(S)$, and hence the above right hand side is in fact equal to the conditional expectation of $\mu(S)$, almost surely (convergence in $L^1$ of a sequence of random variables implies convergence in $L^1$ of its conditional expectations against a fixed $\sigma$-algebra). We deduce
$$
\mu^n(S) = \E( \mu(S) | \cF_n),
$$
almost surely, and hence $\mu'(S) = \mu(S)$, as desired.
\end{proof}

\section{Weak Convergence}

We now finish the proof of Theorem \ref{T:conv} by showing that the sequence of measures $\mu_\eps$ converges in probability for the weak topology towards a measure $\mu$ defined by the limits of quantities of the form $\mu_\eps(S), $ where $S$ is a cube such that $\bar S \subset D$. As the arguments are relatively standard we will be brief. We start by proving that the total mass of the measures $\mu_\eps$ converge in probability. Let $D_n$ be an increasing sequence of domains such that $\cup_n D_n = D$. Let $\ell = \sup_n \mu(D_n)$. Let us show that $\mu_\eps(D) \to \ell$ in probability as $\eps \to 0$. Note first that $\ell < \infty$ a.s., since by monotone convergence $\E(\ell) = \sup_n \E(\mu(D_n)) = \sigma (D) < \infty$ by our assumptions on $\sigma$.

Let $\delta >0$. Then we can write $\mu_\eps(D) = \mu_\eps(D_n) + \mu_\eps(D \setminus D_n) $ for any $n\ge 1$. The first term converges in probability as $\eps \to 0$ to $\mu(D_n)$ by the part of the theorem already proved, and for the second term, $ \E (\mu_\eps( D\setminus D_n)) = \sigma(D\setminus D_n) \le \delta^2$ for some $n$ sufficiently large depending only on $\delta$. Fixing that value of $n$, by Markov's inequality, $\P( \mu_\eps(D \setminus D_n) > \delta) \le \delta$, and for the same reason, $\P( |\ell - \mu(D_n) | > \delta) \le \delta$ as well. Then for $\eps$ small enough (depending only on on $n$ and $\delta$, and thus only on $\delta$) we see that $|\mu_\eps(D_n) - \mu(D_n) |\le \delta$ with probability at least $1- \delta$. Altogether, $ \P( \mu_\eps(D) - \ell |  \ge 2 \delta ) \le 3\delta$. So $ \mu_\eps(D) \to \ell$ in probability as $\eps \to 0$.

Now let us prove weak convergence. Let $\cA$ denote the $\pi$-system of subsets of $\R^d$ of
the form $ A = [x_1, y_1) \times \ldots \times[x_d, y_d)$ where $x_i, y_i \in \Q, 1\le i \le d$ and such that $\bar A \subset D$, and note that the $\sigma$-algebra generated by $\cA$ is the Borel $\sigma$-field on $D$.
We aim to check that for every deterministic sequence $\eps_k$ tending to zero, one can find a further (deterministic) subsequence $\eps'_k$ such that $\mu_{\eps'_k}$ converges almost surely for the weak topology on $D$. Observe that $\mu_\eps(A)$ converges in probability to $\mu (A)$ for any $A \in \cA$, by the part of the theorem which is already proved. Let $\eps_k$ be any sequence tending to zero. Fix any subsequence $\eps'_k$ such that $\mu_{\eps'_k}(A)$ converges to $\mu (A)$ almost surely for all $A \in \cA$ (which is possible since $\cA$ is countable).

Let $A = [x_1, y_1) \times \ldots \times[x_d, y_d)\in \cA$. We first claim that
\begin{equation}\label{pm1}
\mu(A) = \sup_{z_i} \{\mu ( [x_1, z_1) \times \ldots \times[x_d, z_d))   \}
\end{equation}
where the sup is over all $z_i < y_i, z_i \in \Q, 1 \le i \le d$. Indeed, clearly the left hand side is a.s. greater or equal to the right hand side, but both sides have the same expectation (if the Radon measure $\sigma(dx)$ is absolutely continuous with respect to Lebesgue measure -- if not, it may be necessary to translate all the cubes by a fixed independent random variable, say standard Gaussian in $\R^d$). Likewise, it is easy to check that
\begin{equation}\label{pm2}
\mu(A) = \inf_{z_i} \{\mu (  [x_1, z_1) \times \ldots \times[x_d, z_d))   \}
\end{equation}
where now the inf is over all $z_i >  y_i, z_i \in \Q, 1 \le i \le d$.

 Now, since $\mu_\eps(D)$ converges in probability, we can assume without loss of generality that our subsequence $\eps'_k$ is such that $\mu_{\eps'_k} (D)$ converges a.s. as $\eps'_k \to 0$. Hence the measures $\mu_{\eps'_k}$ are tight in the space of Borel measures on $D$ with the topology of weak convergence. Let $\tilde \mu$ be any weak limit.

 Then by the portmanteau theorem together with \eqref{pm1} and \eqref{pm2} it is easy to check that $\tilde \mu(A) = \mu(A)$ for any $A \in \cA$. (The portmanteau is more classically stated for probability measures, but there is no problem in using the theorem here since we already know convergence of the total mass, so we can equivalently work with the normalised measures $\mu_\eps / \mu_\eps(D)$).
 This identifies uniquely the limit $\tilde \mu$, and so in fact $\mu_{\eps'_k}$ converges a.s. weakly on $D$. As discussed already, this implies the weak convergence in probability of the measures $\mu_\eps$ on $D$. \qed

\end{document}